\documentclass[11pt]{article}
\usepackage{amsmath}
\usepackage{amssymb}
\usepackage{graphicx}
\usepackage{hyperref}
\usepackage{cite}
\usepackage{geometry}
\usepackage{amsthm}
\usepackage{csquotes}
\newcommand{\R}{\mathbb{R}}
\newcommand{\E}{\mathbb{E}}
\DeclareMathOperator{\Cov}{Cov}

\newtheorem{theorem}{Theorem}
\newtheorem{remark}{Remark}

\usepackage{color}

\definecolor{lightblue}{rgb}{0.95,0.97,1}
\definecolor{darkblue}{rgb}{0.2,0.4,0.7}

\newcommand{\reviewbox}[2]{%
    \vspace{8pt}
    \noindent\fbox{\colorbox{lightblue}{\parbox{0.92\textwidth}{%
        \vspace{4pt}
        {\color{darkblue}\textbf{#1}}
        \vspace{2pt}
        \hrule height 0.5pt
        \vspace{4pt}
        #2
        \vspace{4pt}
    }}}
    \vspace{8pt}
}

\title{Can LLMs Write Mathematics Papers? A Case Study in Reservoir Computing}
\author{Allen Hart}
\date{September 27, 2025}

\begin{document}

\maketitle

\begin{abstract}
As AI capabilities continue to grow exponentially on economically relevant human expert tasks, with task completion horizons doubling every 7 months according to the Model Evaluation and Threat Research (METR), we are interested in how this applies to the task of mathematics research. To explore this, we evaluated the capability of four frontier large language models (LLMs), ChatGPT 5, Claude 4.1 Opus, Gemini 2.5 Pro, and Grok 4, at the task of creating a mini-paper on reservoir computing. All models produced engaging papers with some apparent understanding of various techniques, but were sometimes lead to mistakes by surface level understanding of key ideas. That said, the capabilities on LLMs on this task was likely as good or greater than that predicted by METR.
\end{abstract}

\section{Introduction}

The rapid advancement of large language models (LLMs) has raised fundamental questions about their potential role in mathematics research. LLMs have demonstrated remarkable capabilities in code generation (see SWE-bench \cite{jimenez2024} leaderboards), mathematical reasoning (including gold medals at the International Mathematical Olympiad \cite{deepmind2025imo,reuters2025imo}), and technical writing. 
METR's \cite{kwa2025metr} recent findings show frontier models can now complete tasks in software engineering, cyber security, general reasoning and machine learning requiring $\sim$1 hour of expert human effort with 50\% reliability, with the time horizon doubling approximately every 7 months since 2019. Similarly, OpenAI's proposed GDP-VAL benchmark aims to measure economically valuable work, with their research suggesting that Claude Opus 4.1 outperforms or equals human experts in 48\% of evaluated tasks. Against this background of extraordinary progress in AI in these tasks, our paper evaluates the performance of leading LLMs on mathematical research, a task involving the synthesis of mathematical reasoning, code implementation, and academic writing.

We asked the LLMs to write papers on reservoir computing \cite{jaeger2001,maass2002}, a machine learning paradigm that offers an interesting test case, largely because it combines mathematical theory (dynamical systems, linear algebra, probability), with numerical experiments.
Our experiment design simulates a realistic research scenario: a supervisor providing a research direction to a student, expecting them to understand the mathematical framework, prove a result (similar to what \cite{feldman2025} call the G\"{o}del test), implement a numerical experiment, and write up the findings in a conventional academic sytle. This approach differs from typical LLM benchmarks that focus on isolated capabilities such as MMLU \cite{hendrycks2021}, HumanEval \cite{chen2021}, or GSM8K \cite{cobbe2021}, instead testing the integration of multiple skills required for mathematical research.

\section{Methods}

\subsection{Task Design}

The task was communicated to the models using three sequential prompts, delivered one at a time with the model required to complete each before receiving the next. This approach is a very simple form of scaffolding, a technique widely used to improve complex task completion \cite{wu2022,suzgun2024}. By breaking down complex tasks into smaller, more manageable subtasks handled sequentially, scaffolding guides LLMs through multi-step reasoning processes that would be difficult to accomplish in a single prompt. Each model was tested in a new conversation instance to avoid contamination from previous interactions. The prompts are as follows:

\textbf{Prompt 1:}
\begin{quote}
Your goal is to write a mini paper on reservoir computing using the papers 

https://arxiv.org/pdf/2211.09515 

and 

https://arxiv.org/pdf/2108.05024 

for context. This specifically involves 3 tasks.

1. Given the reservoir system $X_t = AX_{t-1} + C(\omega\phi^t(m) + \mathcal{N}(0,\sigma^2))$ where $\mathcal{N}(0,\sigma^2)$ is a normal random variable and the remaining terms are those described in equation (1.2) in https://arxiv.org/pdf/2108.05024, derive an expression for the random variable $X_t$ in the limit as $t \to \infty$ as a sum of the generalised synchronisation $f_{(\phi,\omega,F)}$ (equation (1.3) in https://arxiv.org/pdf/2108.05024) and a random variable. I recommend you use similar arguments as seen in theorem 7.1 in https://arxiv.org/pdf/2211.09515 which is set in continuous time.

2. Write python code that reconstructs the experiment described in section 5 of https://arxiv.org/pdf/2108.05024 which is entitled ``the Lorenz system. Forecasting in the presence of noise''. Repeat the experiment with $A,C$ chosen to correspond to the Takens' delay map, as described in the section entitled ``Relation with Takens' Theorem'' in https://arxiv.org/pdf/2108.05024. Now create relevant plots that allow us to compare the results of the Takens delay map to the randomly generated A,C to see if one is better than the other.

3. Create latex file contains includes the mini paper, including short intro, and conclusion, the derivation of the expression for $X_t$ as described in task 1 and plots output by task 2.

We will do each task one at a time. Please think for as long as necessary to output the best first go at task 1.
\end{quote}

\textbf{Prompt 2:}
\begin{quote}
Thank you, Now please think for as long as necessary to output the best first go at task 2.
\end{quote}

\textbf{Prompt 3:}
\begin{quote}
Thank you, Now please think for as long as necessary to output the best first go at task 3.
\end{quote}

The prompts referenced two specific arXiv papers, requiring models to understand and synthesize existing work rather than generate content from scratch. A typographical error (``contains includes'' in task 3 of Prompt 1) was discovered only after the experiment began but was left unchanged for consistency. While unintentional, this error incidentally tests the models' robustness to typos which are common in human communication.

\subsection{Small and Big World Benchmarks}

Our task differs fundamentally from existing mathematical benchmarks like the International Mathematical Olympiad (IMO) problems \cite{imo2024} or physics competitions \cite{ipho2024}, where all information needed to solve problems is self-contained. While models have achieved impressive performance on such benchmarks, including gold medals at IMO \cite{deepmind2025imo,reuters2025imo}, these problems represent what Sutton and colleagues call ``small world'' problems where complete information is available upfront \cite{javed2024}.

In contrast, mathematics research operates in what the Big World Hypothesis describes as environments where ``the world is multiple orders of magnitude larger than the agent'' \cite{javed2024}. Real research requires synthesizing ideas across a vast space of papers and books, identifying relevant connections that may not be obvious, and applying methods from disparate fields. Our task, requiring models to read two technical papers, understand relevant context, and apply methods from linear algebra and probability theory, represents a step toward evaluating these ``big world'' capabilities.

\subsection{Evaluation Framework}

Papers were evaluated across three dimensions, which would be a reasonable rubric for evaluating work produced by a student:
\begin{enumerate}
    \item \textbf{Mathematical Correctness}: Accuracy of the theoretical derivation, proper use of notation, logical flow of arguments
    \item \textbf{Experimental Implementation}: Code functionality, adherence to specifications, parameter choices
    \item \textbf{Writing Quality}: Structure, clarity, appropriate use of academic conventions including proper references
\end{enumerate}

Each paper received detailed reviewer comments that I (the author) integrated directly into the PDF, mimicking academic peer review.

\subsection{Revision Phase}

After initial review, models were given an opportunity to revise their papers. Each model received the same prompt:

\begin{quote}
Thank you, I have reviewed the mini-paper and will return the reviewed paper. Please take as long as you need write a second draft in Latex.
\end{quote}

This revision phase tested models' ability to incorporate feedback and improve their work, a crucial aspect of real research. While not the primary focus of this analysis, the revision responses revealed interesting patterns in how models handle criticism and correction, which we discuss briefly in Section 4.8.

\section{Results}

\subsection{Processing Time and Resources}

Models required approximately 5--10 minutes of processing time per task over the 4 prompts. Without public disclosure of model architectures and sizes, computation resource usage by the models is unknown.

\subsection{Mathematical Derivation Performance}

The models all performed the derivation well. I have included an excerpt from Claude's paper (which I thought had the best first go at this problem) with reviewer comments.

\noindent\textbf{\emph{Claude Opus 4.1:}}
\begin{quote}

\section*{Theoretical Framework}

\subsection*{Reservoir System with Noise}

Consider a discrete-time reservoir system driven by a deterministic dynamical system $\phi: M \to M$ on a manifold $M$, observed through a measurement function $\omega: M \to \R^d$. The reservoir state $X_t \in \R^N$ evolves according to:
\begin{equation}
X_t = AX_{t-1} + C(\omega\phi^t(m) + \xi_t)
\label{eq:reservoir_noise}
\end{equation}
where:
\begin{itemize}
    \item $A \in \R^{N \times N}$ is the reservoir matrix with spectral radius $\rho(A) < 1$
    \item $C \in \R^{N \times d}$ is the input coupling matrix
    \item $\phi^t(m)$ denotes the $t$-th iterate of $\phi$ starting from $m \in M$
    \item $\xi_t \sim \mathcal{N}(0, \sigma^2 I_d)$ represents i.i.d. Gaussian measurement noise
\end{itemize}

\reviewbox{REVIEWER COMMENT 4. Theoretical set up}{Claude makes an interesting choice here to define $\omega : M \to \mathbb{R}^d$ for general dimension $d$ rather than $d = 1$ which is all that is necessary for this paper.}

\subsection*{Generalized Synchronization}

In the absence of noise ($\sigma = 0$), Hart et al. \cite{hart2021} established that under the echo state property, the reservoir state converges to a unique generalized synchronization function:
\begin{equation}
f_{(\phi,\omega,F)}(x) = \sum_{k=0}^{\infty} A^k C \omega(\phi^{k+1}(x))
\label{eq:gen_sync}
\end{equation}
This function satisfies the consistency relation $f_{(\phi,\omega,F)}(\phi(x)) = Af_{(\phi,\omega,F)}(x) + C\omega(\phi(x))$.

\reviewbox{REVIEWER COMMENT 5. Generalised synchronisation}{This is good except for a confusing decision to use $x$ instead of $m \in M$. The echo state property is included without a citation or definition. It is well known but could benefit from at least reference. The inclusion of the consistency relation is nice.}

\subsection*{Asymptotic Behavior with Noise}

We now derive the asymptotic behavior of the reservoir system \eqref{eq:reservoir_noise} in the presence of noise.

\begin{theorem}[Reservoir Decomposition with Noise]
\label{thm:main}
Consider the reservoir system \eqref{eq:reservoir_noise} with $\rho(A) < 1$. As $t \to \infty$, the reservoir state converges in distribution to:
\begin{equation}
X_\infty = f_{(\phi,\omega,F)}(\phi^t(m)) + Z_\infty
\label{eq:decomposition}
\end{equation}
where:
\begin{enumerate}
    \item $f_{(\phi,\omega,F)}(\phi^t(m))$ is the generalized synchronization function evaluated at the current state
    \item $Z_\infty \sim \mathcal{N}(0, \Sigma)$ is a Gaussian random variable independent of the deterministic dynamics
    \item The covariance matrix $\Sigma$ satisfies the discrete Lyapunov equation:
    \begin{equation}
    \Sigma = A\Sigma A^T + \sigma^2 CC^T
    \label{eq:lyapunov}
    \end{equation}
\end{enumerate}
\end{theorem}

\begin{proof}
We proceed by recursive expansion of the reservoir dynamics. From equation \eqref{eq:reservoir_noise}, we have:
\begin{equation}
X_t = A^t X_0 + \sum_{k=0}^{t-1} A^{t-1-k}C(\omega\phi^{k+1}(m) + \xi_{k+1})
\end{equation}

This can be decomposed into three components:
\begin{equation}
X_t = \underbrace{A^t X_0}_{\text{Initial condition}} + \underbrace{\sum_{k=0}^{t-1} A^{t-1-k}C\omega\phi^{k+1}(m)}_{\text{Deterministic drive}} + \underbrace{\sum_{k=0}^{t-1} A^{t-1-k}C\xi_{k+1}}_{\text{Stochastic component}}
\end{equation}

\textbf{Step 1: Initial condition decay.} Since $\rho(A) < 1$, we have $\|A^t\| \to 0$ exponentially, thus:
\begin{equation}
\lim_{t \to \infty} A^t X_0 = 0
\end{equation}

\textbf{Step 2: Deterministic convergence.} The echo state property ensures that the deterministic sum converges absolutely:
\begin{equation}
\lim_{t \to \infty} \sum_{k=0}^{t-1} A^{t-1-k}C\omega\phi^{k+1}(m) = \sum_{k=0}^{\infty} A^k C\omega\phi^{k+1}(\phi^t(m)) = f_{(\phi,\omega,F)}(\phi^t(m))
\end{equation}
where we used the time-invariance of the infinite sum due to the stationarity of the driving dynamics.

\textbf{Step 3: Stochastic convergence.} Define $Z_t = \sum_{k=0}^{t-1} A^{t-1-k}C\xi_{k+1}$. Since $\xi_k$ are i.i.d. Gaussian with zero mean, $Z_t$ is also Gaussian with:
\begin{equation}
\E[Z_t] = 0
\end{equation}

The covariance matrix is:
\begin{align}
\Cov(Z_t) &= \E[Z_t Z_t^T] = \sum_{k=0}^{t-1} A^{t-1-k}C \E[\xi_{k+1}\xi_{k+1}^T] C^T (A^T)^{t-1-k}\\
&= \sigma^2 \sum_{k=0}^{t-1} A^k CC^T (A^T)^k
\end{align}

Since $\rho(A) < 1$, this sum converges as $t \to \infty$:
\begin{equation}
\lim_{t \to \infty} \Cov(Z_t) = \sigma^2 \sum_{k=0}^{\infty} A^k CC^T (A^T)^k = \Sigma
\end{equation}

where $\Sigma$ satisfies the discrete Lyapunov equation \eqref{eq:lyapunov}, which can be verified by:
\begin{align}
\Sigma &= \sigma^2 \sum_{k=0}^{\infty} A^k CC^T (A^T)^k\\
&= \sigma^2 CC^T + \sigma^2 \sum_{k=1}^{\infty} A^k CC^T (A^T)^k\\
&= \sigma^2 CC^T + A \left(\sigma^2 \sum_{j=0}^{\infty} A^j CC^T (A^T)^j\right) A^T\\
&= \sigma^2 CC^T + A\Sigma A^T
\end{align}

Therefore, $Z_t \xrightarrow{d} Z_\infty \sim \mathcal{N}(0, \Sigma)$ as $t \to \infty$, completing the proof.
\end{proof}

\begin{remark}
The decomposition \eqref{eq:decomposition} reveals that noise does not destroy the generalized synchronization property but rather adds a stationary Gaussian perturbation. The noise covariance $\Sigma$ depends on the reservoir architecture through matrices $A$ and $C$, suggesting that different designs may exhibit varying noise sensitivity.
\end{remark}

\reviewbox{REVIEWER COMMENT 6. Generalised synchronisation}{The theorem, proof, and remark seem excellent. The remark was not required by the prompt but it is thought provoking.}

\end{quote}

We can see that Claude synthesised ideas from reservoir computing (the echo state property, and generalised synchronisation) with general techniques in linear algebra and probability to obtain the formula for $X_t$ in the limit. Claude decided (without instruction in the prompt) to state and prove that the covariance matrix $\Sigma$ is the unique solution of a Lyapunov equation. This equation does not appear in either reference paper. The model also provides an unprompted and interesting remark about the noise depending on $A,C$. There is a sense that that Claude is showing some understanding of beyond symbolic manipulation. Now ChatGPT displayed a similar level of understanding of the material but opted to provide a short sketch proof without line by line symbolic manipulation. After my review complaining about this, ChatGPT produced a more detailed derivation including more steps, but some of these steps are very suspect. For instance a term $\eta_t^{(S)}$ was used in the proof but not defined anywhere. We can guess its meaning based on the definition of another term $\eta_t^{(T)}$, but then a limit is taken over both $S$ and $T$ simultaneously. Even if there is some way to make sense of the argument, the arguments given by the other models are much cleaner.

Next, Gemini defined very few of the symbols used in the argument, including failing to define $\omega$, $\phi$, $t$ or $m$, and proceeded to provide a very terse proof that skipped much of the symbolic manipulation. Gemini was able to largely fix these oversights after receiving the review - resulting in a clear proof. The final model Grok provided a detailed proof on the first try with all the steps included, similar to Claude, and was the only model to clearly emphasise that the derivation closely followed that described in one of the reference papers. Grok also refered to $\omega: \mathcal{M} \to \mathcal{R}$ a generic observation function, which is curious because the genericity is not defined or relevant to Grok's mini-paper. Grok likely used the word \emph{generic} here is because the word appears frequently in the reference paper where the property is necessary in that context. This suggests a level of imitation from Grok in place of deeper reasoning. Overall, the models were mostly successful at the derivation, but there was a subtle lack of precision across all the LLMs. For example none of the four machines noted that $\phi$ needed to be invertible for the arguments to go through - which is especially interesting because $\phi$ is always a diffeomorphism in the reference papers, and $\phi^{-j}$ is used in explicitly in the proofs.

\subsection{Experimental Implementation}

There is an important distinction between imitation and reasoning, and the implementation of the experiment revealed a lot of the former from the models. For example ChatGPT and Claude used linear regression to predict the next state of the target dynamical system (the Lorenz system). Using a linear predictor is the standard approach, but the approach will fail if the reservoir system is itself linear, which is the case in the mini-paper. This is because a linear reservoir system must be given a nonlinear predictor to maintain the universal approximation property. Claude even mentions in their paper that ``This approach leverages the universal approximation properties of random projections in high dimensions.'' which is incorrect because Claude used a linear predictor! This consideration of universality from Claude without then integrating the idea into the experimental design is a clear example of imitation. The decision to use a linear regressor is especially interesting because it explicitly deviates from the experiment described in the reference paper that the prompt suggested the models follow. This suggests the models prefer to repeat the standard approach over carefully reading the reference paper. 

Curiously, Gemini simply used a different and more common reservoir system altogether (an Echo State Network) instead, without elaboration. This choice disconnects the theory from the experiment suggesting the Gemini does not have a coherent vision of the purpose of the paper. ChatGPT did mention in the conclusion that using a neural network as a predictor could give better predictions than a linear regressor - which could be suggestive of some understanding - but is likely a lucky guess given the statement is true in many other contexts. Notably, Grok was the only model to correctly implement a neural network readout as specified in Section 5 of ``Learning strange attractors with reservoir systems'', using a deep network with 10 hidden layers and scaled sigmoid activations. Now considering overall performance across the four models, there were several experimental decisions which some models made but did not describe correctly in their paper, many of which are included in the review comments in the supplementary material. Despite some evidence of surface level mathematical understanding, and significant conceptual mistakes, all models generated working code without any need for debugging, outputting well formatted figures.

\subsection{Response to reviews}

There was a tendency in some models to overcorrect in response to reviews. For example I suggested to Grok that they connect their findings to the broader literature, which was achieved to some extent, but resulted in humerous additions to the paper like ``\emph{This mini-paper
connects to broader literature on embeddings in reservoir computing}''. There is an interesting sense here that because the AIs are presented in the from of chatbots they are highly motivated to respond strongly to the latest prompt, even at the expense of the original goal of writing a good paper. 

A notable shortcoming of both Grok and Gemini in their first draft was that the presentation was very unlike a paper, with Gemini providing no references at all and Grok providing references in the form of URL links. Both models were able to improve the formatting dramatically when prompted by the reviews to do so. An interesting quirk of the Gemini paper was that the model listed itself as the sole author.

Both Claude and ChatGPT initially produced references in the standard academic format but fabricated the author list, dates, and paper titles. After review both LLMs corrected this. This was very intriguing, because the models are intelligent enough to understand that references should not be fabricated, and are capable, if given enough compute, to produce correct ones. This is highly suggestive of an alignment failure, where both Claude and ChatGPT reasoned that producing accurate references was not worth the compute, and it was preferable to fake references. Work by OpenAI \cite{arxiv250915541} and Anthropic \cite{alignmentfaking2025} have shown LLMs have considerable self awareness and consider deceptions of this kind very frequently.

\section{Discussion}

\subsection{Mathematics Research and GDP-VAL}

The GDP-VAL tasks are economically relevant tasks that take between 4 and 7 hours for a human expert. A suite of LLMs, including the four tested in this paper, tested on these tasks met or exceeded human performance in 24-48\% cases, with Claude Opus 4.1 the best performing model achieving the score of 48\%. Thus, if the tasks in GDPeval generalise to the task of mathematics research, we would expect the papers produced by the models to be at the standard of a human mathematician given a time limit of 4-7 hours. This sounds very plausible to me, but is very hard to test explicitly. To design such an experiment there would need to be something like a review panel to assess the quality of the papers produced by both human experts and AIs for comparison. There is also the tricky question of subject matter expertise - for instance I am an author on both reference papers which would give me a considerable advantage over an expert in reservoir computing who has not read either paper, who in turn has a considerable advantage over a mathematician in another field.

\subsection{Limitations and Future Work}

This study has several limitations that should guide future work. First, the assessment and task design were created by a single mathematician (the author), which introduces subjective bias into both the rubric and the interpretation of results. A broader panel of evaluators with different backgrounds and standards would provide a more robust and reliable assessment framework. 

Second, the scope of mathematics considered here was very narrow: reservoir computing and its surrounding dynamical systems theory. Though I think it is a reasonable test case, the area was selected because it happens to be my own research area. Future evaluations should include a wider range of mathematical domains - ideally encompassing different styles, cultures and conventions from algebra and topology to as statistics and optimization. The level of difficulty and complexity of tasks should be explored further as well, especially as models improve.

Third, the experimental setup relied on a simple sequential prompting structure with three stages. More complex scaffolding approaches—such as agentic prompting, allowing models to retry or refine their answers, or using multi-agent collaborations—may reveal higher capabilities. Evaluating the effect of these different modes of interaction on mathematical performance remains an important open question.  

\section{Conclusion}

Research mathematics is not a ``small world'' problem like those represented in most existing benchmarks, but a ``big world'' problem \cite{javed2024}, requiring synthesis of ideas across large and diverse domains. In this study, we tested the capabilities of four frontier LLMs, ChatGPT 5, Gemini 2.5 Pro, Claude Opus 4.1, and Grok 4, on the big world task of writing a research-style mini-paper. The task required deriving a new result, implementing and running a related numerical experiment, and producing a coherent academic write-up. 

The models produced papers that were mostly clear, mathematically structured, and accompanied by functional code, though all contained flaws of precision and understanding. Importantly, the quality of the outputs was not far from what one might expect from a human researcher working only a few hours on the same task. This places the models’ current performance on research mathematics roughly on trend with, or slightly exceeding, the scaling predictions made by GDPeval and METR.  

The scope of the paper was narrowly defined, and the models demonstrated only limited ability to exercise what might be called ``research taste'' - the ability to formulate novel questions or select promising directions. A true artificial general intelligence would likely demonstrate such qualities by autonomously selecting its own research problems and cultivating its own interests. That challenge remains for future generations of models, but the progress already observed suggests that the transition from competent assistant to autonomous researcher may not be far away.

\bibliographystyle{unsrt}
\bibliography{references}

\end{document}